\documentclass{amsart}
\usepackage{color}

\newtheorem{theorem}{Theorem}[section]
\newtheorem{lemma}[theorem]{Lemma}

\theoremstyle{definition}
\newtheorem{definition}[theorem]{Definition}

\theoremstyle{remark}
\newtheorem{remark}[theorem]{Remark}
\newtheorem{note}[theorem]{Note}

\numberwithin{equation}{section}

\newcommand{\N}{\mathbb{N}}
\newcommand{\R}{\mathbb{R}}
\newcommand{\cC}{\mathcal{C}}
\newcommand{\cF}{\mathcal{F}}
\DeclareMathOperator{\Lip}{Lip}
\newcommand{\be}{\begin{equation}}
\newcommand{\ee}{\end{equation}}

\begin{document}

\title{Dimension preserving approximation}

\author{S. Verma}
\address{Department of Mathematics, IIT Delhi, New Delhi, India 110016 }
\email{saurabh331146@gmail.com}

\author{P.R. Massopust}
\address{Centre of Mathematics, Technical University of Munich, Boltzmannstr. 3,
85748 Garching b. Munich, Germany }
\email{massopust@ma.tum.de}



%\thanks{The author was supported in part by UGC.}

%    General info

%\date{January 1, 2001 and, in revised form, June 22, 2001.}

 

\begin{abstract}
This article introduces the novel notion of dimension preserving approximation for continuous functions defined on $[0,1]$ and initiates the study of it. Restrictions and extensions of continuous functions in regards to fractal dimensions are also investigated.
\end{abstract}

\maketitle

%%%%%%%%%%%%%%%%%%%%%%%%%%%%%%%%%%%%%%%%%%%%%%%%%%%%%%%%%%%%%%%%%%%%%%%%

%%%%%%%%%%%%%%%%%%%%%%%%%%%%%%%%%%%%%%%%%%%%%%%%%%%%%%%%%%%%%%%%%%%%%%%%
.

\section{Introduction} 
We start our discussion with a function $g\in \cC[0,1] := \{f:[0,1]\to \R : \text{$f$ is continuous on $[0,1]$}\}$ with $  \dim G_{g} > 1$. Here and in the following, we denote the graph of a function $g$ by $G_g$. In the present informal discussion, we use $\dim$ to denote a fractal dimension. For the existence of such functions $g$, see, for instance, \cite{Shen}.

The function $f:[0,1] \to \mathbb{R}$ defined by $f(x) :=\int\limits_{0}^{x} g(t)dt$ will have the following properties: 
\[
\dim G_{f} =1\quad\text{and}\quad\dim G_{f'} =\dim G_{g} > 1.
\]
If we approximate $f$ by Bernstein polynomials $B_n(f)$ of order $n$,
\[
B_n(f)(x) := \sum_{k=0}^n \binom{n}{k} f\left(\frac{k}{n}\right)\, x^k (1-x)^{n-k},  
\]
then $B_n(f)$ converges uniformly to $f$ and $(B_n(f))'$ converges uniformly to $f'=g$. (We refer the interested reader to \cite{Gal} for Bernstein polynomials and their properties.) Note that $(B_n(f))'$ is again a polynomial and, thus, the fractal dimension of $(B_n(f))'$ is equal to one. The above conveys that the approximation of a function by Bernstein polynomials preserves the function class but not the dimension of its derivative. 

The current article targets to study approximation aspects with respect to fractal dimensions of a function and its derivative.

The structure of this paper is as follows. After a brief introduction to fractal dimensions in Section 2, the novel concept of dimension preserving approximation is introduced in Section 3 and some of its properties are discussed. Section 4 deals with the restriction and extension of continuous functions in regards to fractal dimensions.
 
\section{Hausdorff dimension, box dimension, and packing dimension}

In this section, we introduce those fractal dimensions that are relevant for the present paper. These are the Hausdorff dimension, the box dimension, and the packing dimension defined for nonempty subsets of a separable metric space $(X,d_X).$ For more details about these fractal dimensions and for proofs, we refer the interested reader to, for instance, \cite{Falc3,Fal,PM1}.

To this end, let $(X,d_X)$ be a separable metric space. For a non-empty subset $U$ of $X,$ the diameter of $U$ is defined as \[
|U| := \sup\{d_X(x,y): x,y \in U\}.
\]
Let $F$ be a subset of $X$ and $s$ a non-negative real number. The $s$-dimensional Hausdorff measure of $F$ is defined by 
\[
H^s(F) := \lim_{\delta \rightarrow 0^+} \inf \left\{\sum_{i=1}^{\infty}|U_i|^s:  F \subseteq \bigcup\limits_{i=1}^{\infty} U_i~~\text{and }~ |U_i| < \delta \right\},
\]
where the infimum is taken over all countable covers $\{U_i\}_{i\in\N}$ of $F$ by sets $U_i \subset F$ with $|U_i| < \delta$.
\begin{definition}
Let $F \subset X$ and let $s \ge 0.$ The Hausdorff dimension of $F$ is defined by
\[
\dim_H F :=\inf\{s:H^s(F)=0\} =\sup\{s:H^s(F)=\infty\}.
\]
\end{definition}
The Hausdorff dimension satisfies the countable stability property: Let $\{X_i\}_{i\in I}$ be a countable family of sets. Then
\be\label{stab}
\dim_H \left(\bigcup_{i\in I} X_i\right) = \sup\limits_{i\in I}\left\{\dim_H X_i\right\}.
\ee
\begin{definition} 
Let $F$ be any non-empty bounded subset of $X$ and let $N_{\delta}(F)$ be the smallest number of sets of diameter at most $\delta$ which can cover $F.$ The lower and upper box dimensions of $F$ are defined as 
\[
\underline{\dim}_B F :=\varliminf_{\delta \rightarrow 0^+} \frac{\log N_{\delta}(F)}{- \log \delta}
\]
and
\[
\overline{\dim}_B F :=\varlimsup_{\delta \rightarrow 0^+} \frac{\log N_{\delta}(F)}{- \log \delta},
\]
respectively. If the above two expressions are equal, their common value is called the box dimension of $F$:
\[
\dim_B F :=\lim_{\delta \rightarrow 0^+} \frac{\log N_{\delta}(F)}{- \log \delta}.
\]
\end{definition}

Let us introduce a few notions that will lead to the definition of packing dimension. Let $s \ge 0$ and $\delta > 0$. We denote by
\begin{align*}
\mathcal{P}^s_{\delta} (F) := \sup\Bigg\{ & \sum_{i=1}^{\infty} |B_i|^s: \{B_i\} ~\text{is a collection of countably many} \\ 
& \text{disjoint balls of radii at most $\delta$ with centres in $F$} \Bigg\}.
\end{align*}  
It is observed that $\mathcal{P}^s_{\delta} (F)$ decreases with $\delta.$ This further implies that the limit $$ \mathcal{P}^s_{0} (F)= \lim_{\delta \to 0^+} \mathcal{P}^s_{\delta} (F)$$ exists. As $\mathcal{P}^s_{\delta}$ is only a pre-measure, one defines  
\begin{align*}
\mathcal{P}^s (F) :=\inf \Bigg\{\sum_{i=1}^{\infty} \mathcal{P}^s_{0} (F_i): F \subseteq \bigcup\limits_{i=1}^{\infty} F_i \Bigg\}.
\end{align*}
Thus, the packing measure $\mathcal{P}^s$ of $F$ is the infimum of the packing pre-measures $\mathcal{P}^s_0$ of countable covers of $F$.

\begin{definition}
Let $F \subset X$ and $s \ge 0.$ The packing dimension of $F$ is defined by 
\[
\dim_P F :=\inf\{s:\mathcal{P}^s(F)=0\}=\sup\{s:\mathcal{P}^s(F)=\infty\}.
\]
\end{definition}

It is known that the following inequalities hold between these types of fractal dimensions \cite{Falc3}:
\[
\dim_H F \leq \underline{\dim}_B F \leq \overline{\dim}_B F
\]
and
\[
\dim_H F \leq \dim_P F  \leq \overline{\dim}_B F.
\]

Although there are several other notions of fractal dimension, this article will deal only with those that were introduced above.

\section{Dimension preserving approximation }

In this section, we present some results relating to the invariance of fractal dimensions under certain maps. In what follows, 
let $(X,d_X)$ be a separable metric space, and $(Y,d_Y)$ be a separable normed linear space. We equip the space $X\times Y$ with a metric $d$ defined by
\[
d\big((x,y),(x',y')\big) := \sqrt{d_X(x,x')^2 +d_Y(y,y')^2}.
\]
The Lipschitz constant of a map $f:X\to Y$ is given by
\[
\Lip (f) = \sup_{x,x'\in X, x \neq x'} \frac{d_Y\big(f(x),f(x')\big)}{d_X(x,x')}.
\]
A map $f$ is said to be \emph{Lipschitz} if $\Lip (f) < + \infty$.

The following result is a generalization of Theorem 1 in \cite{RD}.
\begin{lemma}\label{lipdim}
Let $g:X \rightarrow Y$ be a continuous map between metric spaces $(X,d_X)$ and $(Y,d_Y)$. For a fixed Lipschitz map $f:X \rightarrow Y$, we have that 
\[
\dim_H G_{f+g} = \dim_H G_g\quad\text{and}\quad\dim_P G_{f+g} = \dim_P G_g.
\]
 \end{lemma}
 \begin{proof}
  We define a map $T_f : G_g \rightarrow G_{f+g} $ by $T_f((x,g(x))) :=(x,f(x)+g(x))$, $x\in A$. It is easy to check that the map $T_f$ is onto. Now,
\begin{align*} 
      d\big(T_f((x,g(x))),T_f((y,& g(y)))\big) = d\big((x,f(x)+g(x)),(y,f(y)+g(y))\big)\\ =&\sqrt{d_X(x,y)^2+d_Y\big(f(x)+g(x),f(y)+g(y)\big)^2}\\
       \le &\sqrt{d_X(x,y)^2+2 d_Y(f(x),f(y))^2+2 d_Y(g(x),g(y))^2}\\
       \le & \sqrt{d_X(x,y)^2+2L^2d_X(x,y)^2+2 d_Y(g(x),g(y))^2}\\  \le & M\sqrt{d_X(x,y)^2+d_Y(g(x),g(y))^2}\\
        =& M d\big((x,g(x)),(y,g(y))\big),
\end{align*}
 where $L$ is the Lipschitz constant of $f$ and $M := \max\{ \sqrt{1+2L^2},\sqrt{2}\}.$

On the other hand,
\begin{align*} 
       d\big(T_f((x, & g(x))),T_f((y,g(y)))\big) = d(\big(x,f(x)+g(x)\big),\big(y,f(y)+g(y)\big))\\ =&~\sqrt{d_X(x,y)^2+d_Y\big(f(x)+g(x),f(y)+g(y)\big)^2}\\
               = &~ \frac{M}{M} \sqrt{d_X(x,y)^2+d_Y\big(f(x)+g(x),f(y)+g(y)\big)^2}\\
               \ge &~ \frac{1}{M} \sqrt{d_X(x,y)^2 (1+2L^2)+2 d_Y\big(f(x)+g(x),f(y)+g(y)\big)^2}\\
               \ge &~ \frac{1}{M} \sqrt{d_X(x,y)^2+2d_Y\big(f(x)+g(x),f(y)+g(y)\big)^2+2 d_Y(f(x),f(y))^2}\\
                \ge &~ \frac{1}{M} d\big((x,g(x)),(y,g(y))\big).
\end{align*}
Therefore, $T_f$ is a bi-Lipschitz map. Since the Hausdorff dimension and packing dimension are Lipschitz invariant (see, for instance, \cite{Fal}), we have that $\dim_H G_{f+g} = \dim_H T_f(G_f) = \dim_H G_f$ and $\dim_P G_{f+g} = \dim_P T_f(G_f) = \dim_P G_f.$  
 \end{proof}

\begin{remark}
Since the upper and lower box dimensions and the box dimension (if it exists) are also Lipschitz invariant (cf. \cite{Fal}), the previous lemma also holds for these fractal dimensions.
\end{remark}

It is well-known that the set of Lipschitz functions $[0,1]\to\R$, which we denote by $\mathcal{L}ip [0,1]$, is a dense subset of $\mathcal{C}[0,1]$ when the latter is endowed with the supremum norm $\|\cdot\|_{\infty}.$ We use this fact to prove the following theorem.

\begin{theorem}\label{densethm}
Let $1 \leq \beta \leq 2$. Then the set  $S_{\beta}:=\{f\in  \mathcal{C}[0,1]: \dim G_{f} = \beta\}$ is dense in $\mathcal{C}[0,1].$
\end{theorem}
\begin{proof}
   Let $f\in \mathcal{C}[0,1].$ From the density of $\mathcal{L}ip [0,1]$ in $\mathcal{C}[0,1]$, there exists a sequence $(g_k)$ in $\mathcal{L}ip [0,1]$ which converges to $f$ uniformly. Now let $g \in \mathcal{L}ip [0,1]$ be arbitrary but fixed and fix an $h \in S_{\beta}$. We define a sequence $(f_k)$ by $f_k= g+ \frac{1}{k}h.$ Since $g$ is a Lipschitz function, Lemma \ref{lipdim} implies that $f_k \in S_{\beta}.$ With the convergence of $(f_k)$ to $g$, a basic real analysis result completes the proof.
\end{proof}   

It is known that the box dimension, Hausdorff dimension and packing dimension are also Lipschitz invariant and therefore the above theorem is also valid for these dimensions.

For the next result, we require the following definition.
\begin{definition}(\cite{Aubin})
Let $ T: (X,d_X) \to (Y,d_Y) $ be a set-valued map between two metric spaces.
\begin{enumerate}
\item $T$ is called lower semicontinuous at $x\in X$ if for any open set $U$  in $Y$ such that $ U \cap T(x) \neq \emptyset $ there exists a $\delta > 0$ satisfying $U \cap T(x') \neq \emptyset $ whenever $d_X(x,x') < \delta.$
The map $T$ is called lower semicontinuous if it is lower semicontinuous at every $x \in X.$ 
\item $T$ is said to be closed if the graph of $T$ defined by $G_T:=\{(x,y):y \in T(x)\}$ is a closed subset of $X \times Y$.
\end{enumerate}
\end{definition}

\begin{theorem}
The set-valued function $D:[1,2] \rightarrow \mathcal{C} [0,1]$ defined by 
\[
D(\beta) :=\{f \in \mathcal{C} [0,1]: \dim G_f =\beta \}=S_{\beta}
\] 
is lower semicontinuous. 
\end{theorem}
\begin{proof}
Let $\beta \in [1,2]$ and $U$ be any open set such that $D(\beta) \cap U \ne \emptyset.$ Since $S_{\beta}$ is dense in $\mathcal{C} [0,1]$, we obtain $D(\alpha) \cap U \ne \emptyset, ~\forall \alpha \in [1,2],$ establishing the proof.
\end{proof}

\begin{remark}
The set-valued map $D$ is not closed. If we choose a sequence of polynomials $(p_n)$ converging to a Weierstrass-type nowhere differentiable function $f$ with Hausdorff dimension $>1$ (for examples of such functions, see, e.g., \cite{RD}) then $(1, p_n) \in G_D$ and $(1, p_n) \rightarrow (1,f)$ but $\dim_H(G_f) > 1.$ Therefore, we deduce that $G_D$ is not closed.   
\end{remark}

The following result is well-known in analysis but repeated for the sake of completeness.      
\begin{theorem}[\cite{Rudin}]\label{Rudinthm}
             Let $\big(f_n\big)$ be a sequence of differentiable functions  on $[0,1]$. Assume that the sequence $\big(f_n(x_0)\big)$ converges for some $x_0 \in [0,1].$ If $(f'_n)$ converges uniformly on $[0,1],$ then $\big(f_n\big)$ converges uniformly on $[0,1]$ to a function $f$, and
             $$f'(x)=\lim_{n \to \infty }f'_n(x),$$ for every $x \in [0,1].$
             \end{theorem}
             Note that if $f$ is a continuously differentiable function, then $\dim(G_{f})=1.$ However, we cannot say anything about the dimension of its derivative. For example, take a Weierstrass-type nowhere differentiable continuous function $g:[0,1] \to \mathbb{R}$ as in, for instance \cite{Shen}, with $1\le  \dim G_{g}\le  2$. Then the function $f$ defined by $f(x)=\int\limits_{0}^{x} g(t)dt$ satisfies the following conditions: $\dim G_{f} =1$ and $1\le  \dim G_{f'} =\dim G_{g}\le 2.$ Moreover, we emphasize the fact that functions $f$ defined by an integral formula are always absolutely continuous. Hence, for such functions $f$ we have $\dim G_{f} =1$.
             \begin{theorem}\label{mainthm}
              Suppose $f$ is a continuously differentiable function with $\dim G_{f'} =\beta$ for some $1 \le \beta \le 2.$ Then there exists a sequence of continuously differentiable functions $(f_n)$ satisfying $\dim G_{f_n'} =\beta$, and $(f_n)$ converges uniformly to $f$. 
             \end{theorem}
             \begin{proof}
             From Theorem \ref{densethm} we obtain a sequence of continuous functions $(g_n)$ with $\dim G_{g_n} =\beta$, which converges uniformly to $f'$. 
             Define a function $f_n:[0,1] \to \mathbb{R}$ by $f_n(x)=\int\limits_{0}^{x} g_n(t)dt.$ Then, $f_n'=g_n$ and $(f_n')$ converges to $f'.$ Moreover, one verifies that the sequence $\big(f_n(0)\big)$ converges to zero. In view of Theorem \ref{Rudinthm}, the sequence $(f_n)$ converges uniformly to $f$ with the required condition $\dim G_{f_n'} =\beta$.
             \end{proof}
      \begin{remark}
      The above theorem can be extended as follows.
      Suppose $f$ is a $k-$times continuously differentiable function with $\dim G_{ f^{(k)} } =\beta$ for some $1 \le \beta \le 2.$ Then there exists a sequence of $k-$times continuously differentiable functions $(f_n)$ satisfying $\dim G_{ f_n^{(k)} } =\beta$, which converges uniformly to $f$.
      \end{remark}
      The next theorem deals with both dimension preserving and shape preserving approximation of a continuous function.
      \begin{theorem}\label{mainthm1}
                    Suppose $f$ is a continuously differentiable function with $\dim G_{f'} =\beta$ for some $1 \le \beta \le 2$ and $f(x)\ge 0, \forall x \in [0,1].$ Then there exists a sequence of continuously differentiable functions $(f_n)$ satisfying $\dim G_{f_n'} =\beta$ and $f_n(x)\ge 0, \forall x \in [0,1]$, and $(f_n)$ converges uniformly to $f$. 
      \end{theorem}
      \begin{proof}
      The proof uses arguments similar to those given in Theorems \ref{densethm} and \ref{mainthm}, and is omitted.
      \end{proof}
      
      \subsection{Construction of dimension preserving approximants}
Hutchinson constructed parametrized curves in \cite{H} and Barnsley \cite{MF1} used iterated function system (IFSs) to define a class of functions called fractal interpolation functions (FIFs). A FIF is a continuous function whose graph is the (attractor) invariant set of a suitably chosen IFS. For the benefit of the reader, we briefly revisit the construction of a fractal interpolation function. For material about IFSs and FIFs, we refer the interested reader to, e.g., \cite{MF5,PM1}.

To this end, let $(X,d_X)$ be a complete metric space and let $f:X\to X$. The map $f$ is said to be a \emph{contraction (on $X$)} if $\Lip (f) < 1$.

\begin{definition}
Let $(X,d_X)$ be a complete metric space and let $\cF:=\{f_1, \ldots, f_n\}$ be a finite set of contractions on $X$. Then the pair $(X,\cF)$ is called an iterated function system on $X$.
\end{definition}

\begin{definition}
A nonempty compact subset $K$ of $X$ is called an invariant set or an attractor of the IFS $(X, \cF)$ if it satisfies the self-referential equation
\be\label{self}
K = \bigcup_{i=1}^n f_i (K).
\ee
\end{definition}
\noindent
It can be shown that if such a set $K$ exists, it is unique.

      Let a set of interpolation points $\{(x_i,y_i) : i=0,1,2,...,N\} \subset\R^2$ with increasing abscissae $0 =: x_0 < x_1<x_2< \dots<x_N:=1$ be given. Set $J := \{1,2,...,N-1,N\}$, $I :=[0,1]$ and $I_i := [x_{i-1}, x_{i}]$, $i\in J.$  Let $L_i: I \to I_n$ be affine functions such that $L_i(x_0)=x_{i-1}$ and $L_i(x_N)=x_{i}$ for $i \in J$. Suppose that $F_i: I \times \mathbb{R} \to \mathbb{R}$ are functions that are continuous in the first variable and contractive in the second variable such that
\be\label{Fn}
F_i(x_0,y_0)=y_{i-1}, \quad F_i(x_N,y_N) =y_{i}.
\ee
      Define $$w_i(x,y) := \Big(L_i(x),F_i(x,y) \Big), \quad i \in J,$$
      and consider the IFS $\mathcal{W}=(I \times \mathbb{R}, w_i: i \in J).$
      
      \begin{theorem}[\cite{MF1}] \label{BT1}
      Let $\mathcal{W}$ be the IFS defined above. Then $\mathcal{W}$ has a unique attractor $G= G_f$ which is the graph of a continuous function $f: I \to \mathbb{R}$. Moreover, $f$ interpolates the data set $\{(x_i,y_i) : i\in J\} $, that is, $f(x_i)=y_i$ for all $i\in\{0,1,\ldots, N\}.$
      \end{theorem}
The function $f$ in the above theorem whose graph is the attractor of an IFS is termed a fractal interpolation function. Main features of FIFs are that their graphs are self-referential in the sense of \eqref{self} and that they usually have non-integral box or Hausdorff dimension.

For a special choice of mappings $F_i$, namely, $F_i (x,y) := c_i x + d_i + \alpha_i y$, where the coefficients $c_i$ and $d_i$ are determined by the conditions \eqref{Fn}, and the $\alpha_i \in (-1,1)$ are free parameters, the resulting FIF is called \emph{affine}.

Estimates for the Hausdorff dimension of an affine FIF were presented in \cite{MF1} and also in \cite{Falc2}. The box dimension of classes of affine FIFs was computed in \cite{MF2,MF6,Hardin} and for FIFs generated by bilinear maps in \cite{MF4}. In \cite{HM}, a formula for the box dimension of FIFs $\R^n\to\R^m$ was derived.

In \cite{M2,M1,VCN}, the idea of fractal interpolation was explored further leading to a class of fractal functions associated with a given (classical) function $f\in \mathcal{C}(I)$ as follows. (See also, \cite{PM1} for a similar approach.) 

Let $\Delta:=(x_0,x_1, \dots,x_N)$ be a partition of $I :=[0,1]$ such that, without loss of generality,  $0=x_0<x_1<\dots< x_N=1$. For $i \in J$,  let $L_i :I\to I_i$ be affine (see above) and $F_i : I\times\R\to\R$ be given by
\[
F_i(x,y) := \alpha_iy + f \big( L_i(x) \big)-\alpha_i b(x),
\]
where $b \neq f $ is any continuous function satisfying
\[
b(x_0) = f(x_0), \quad b(x_N) = f(x_N),
\]
and $\alpha:= (\alpha_1, \alpha_2, \dots, \alpha_N) \in (-1,1)^N$. The corresponding FIF, denoted by $f_{\Delta,b}^\alpha$, is called an \emph{$\alpha$-fractal function}. In \cite{M1}, it is noted that $\alpha$-fractal functions satisfy the self-referential equation
      \begin{equation}\label{eqn2}
      f_{\Delta,b}^\alpha(x) = f(x) + \alpha_i (f_{\Delta,b}^{\alpha}- b)\big(L_i^{-1}(x)\big), \quad\forall~~ x \in I_i,~~ i \in J.
      \end{equation}
The following result is a special case of Theorem 3 in \cite{MF6} applied to Lipschitz functions. (See, also \cite[Corollary 5.1]{Akhtar}.)
      \begin{theorem}\label{Akhtar}
       Let $\Delta=(x_0,x_1,\dots , x_N)$ be a partition of $I=[x_0,x_N]$ satisfying $x_0<x_1< \dots < x_N$ and let $ \alpha =(\alpha_1,\alpha_2, \dots ,\alpha_{N}) \in (-1,1)^N.$ Assume that $f$ and $b$ are Lipschitz functions defined on $I$ with $b(x_0)=f(x_0)$ and $b(x_N)=f(x_N).$ If the data points $ \{(x_i, f(x_i)): i =0,1 \dots, N\}$ are not collinear, then
      \begin{equation*}
      \dim_B G_{f_{\Delta,b}^\alpha} =
                         \begin{cases} D,  \text{ if $\sum\limits_{i=1}^{N} |\alpha_i| > 1$};\\
                            1, \text{ otherwise,}
                        \end{cases}
                          \end{equation*}
      where $D$ is the unique positive solution of $\sum\limits_{i=1}^{N} |\alpha_i|a_i^{D-1}=1$. Here, $G_{f_{\Delta,b}^\alpha}$ denotes the graph of $f_{\Delta,b}^\alpha$.
      \end{theorem}
      
      \begin{note}\label{note1}
      We define the second modulus of smoothness with step-weight function $\phi(x):=\sqrt{x(1-x)}$ by 
      $$\omega_{\phi}(f;\delta)=\sup_{0 \le t \le \delta}\sup_{x}|f(x-t\phi(x))- 2 f(x) + f(x+t \phi(x))|,$$
      where the second supremum is taken over those values of $x$ for which every argument belongs to the domain $[0,1].$ In \cite{Totik} the following estimate was  proved:
\[
\|B_n(f) -f\|_{\infty} \le C ~ \omega_{\phi}\Big(f; \frac{1}{\sqrt{n}}\Big),
\]
      for some constant $C>0$. Here, $B_n:\cC(I)\to \Pi_n$ denotes the $n$-th order Bernstein operator and $\Pi_n$ the space of polynomials of degree $\leq n$.
      \end{note}
      Now we are ready to prove the next result.
      \begin{theorem}\label{mainthm3}
      Let $f \in \mathcal{C}(I)$ and $\beta \in (1,2)$. Then there exists a sequence $(f_n)$ of fractal functions converging uniformly to $f$ and $\dim_B G_{f_n} = \beta.$
      \end{theorem}
      \begin{proof}
        For a given $f \in \mathcal{C}(I)$ and $ \beta \in (1,2)$, we choose the partition $\Delta= (0, \frac12, 1)$ of $I=[0,1]$ and a scale vector $ \alpha =(\alpha_1,\alpha_2) \in (-1,1)^2$ by
\[
\alpha_1=\alpha_2 \quad\text{and}\quad\beta = 2+\frac{\log (|\alpha_1|)}{\log 2}.
\]
Further assume, without loss of generality, that the sampling points in $\big\{\big(x_i,f(x_i)\big): i=0,1,2\big\}$ corresponding to $f$ are not collinear. Let $(p_n)_{n \in \mathbb{N}}$ be the sequence of Bernstein polynomials $p_n = B_n(f)$ that converges uniformly to $f$. For each fixed $n \in \mathbb{N}$, construct the $\alpha$-fractal function $(p_n)_{\Delta, B_n(p_n)}^\alpha$ corresponding to $p_n$ by choosing the parameter function $b$ (see above) as $B_n(p_n)$. In the light of Equation \eqref{eqn2} and Note \ref{note1}, a simple and straightforward calculation produces
       \begin{equation*}
       \begin{split}
       \| f- (p_n)_{\Delta, B_n(p_n)}^\alpha\| _\infty \le &~ \|f-p_n \|_\infty + \|p_n - (p_n)_{\Delta, B_n(p_n)}^\alpha\|_\infty \\
       \le &~ \|f -p_n \|_\infty + \frac{|\alpha_1|}{1-|\alpha_1|} \| p_n - B_n(p_n)\|_\infty\\
       \le &~  C ~ \omega_{\phi}\Big(f; \frac{1}{\sqrt{n}}\Big) + \frac{C ~|\alpha_1|}{1-|\alpha_1|}  ~ \omega_{\phi}\Big(p_n; \frac{1}{\sqrt{n}}\Big).
       \end{split}
       \end{equation*}
We therefore conclude that the sequence $(p_n)_{\Delta, B_n(p_n)}^\alpha $ converges uniformly to $f$. For each fixed $n \in \mathbb{N}$, the functions $p_n$ and $B_n(p_n)$ are Lipschitz continuous and the set of data points $\{(x_i, p_n(x_i)): i=0,1,2\}$ is not collinear. Hence, with the help of Theorem \ref{Akhtar}, the box dimensions of the graphs of $(p_n)_{\Delta, B_n(p_n)}^\alpha$, which depend only on the partition and the scaling vector, are all same and are equal to $\beta.$
       \end{proof}
       \begin{remark}
       The previous theorem also determines the order of approximation by fractal functions. More precisely, for a given function $f \in \mathcal{C}(I)$ and $\beta \in (1,2)$ we have the following estimate
       \[
       \|f -f_n\|_{\infty} \le C ~ \Bigg[\omega_{\phi}\Big(f; \frac{1}{\sqrt{n}}\Big) +\omega_{\phi}\Big(B_n(f); \frac{1}{\sqrt{n}}\Big)\Bigg],
       \]
       where $(f_n)$ is a sequence of fractal functions as in the above theorem and constant $C$ depends only on $\beta.$ The above order of approximation is not claimed to be the optimal. Note that  though there are many other approximation polynomials, so called Bernstein-type polynomials, we have used only the Bernstein polynomials in the previous theorem. The reader is encouraged to consult \cite{Gal} for a more detailed study on order of convergence by Bernstein-type polynomials.
       \end{remark}
       Next, we approximate a given function by a sequence of fractal functions having the same Hausdorff dimension. For this purpose, we need to quote the following result can be found in \cite{Barany2} and is based on work presented in \cite{Barany1}.
       \begin{theorem}[\cite{Barany2}, Theorem $2.1$]\label{Barany}
       Let the data set $\triangle=\{(x_i,y_i) \in I \times \mathbb{R}:i=1,2,\dots ,m\}$ be given so that $0=x_0 < x_1< \dots
       < x_m=1$. Assume that $\sum\limits_{i=1}^{m}|\alpha_i| > 1$ and that there exists an $i \ne j$ such that 
       \begin{equation}\label{condi}
       \frac{y_i -y_{i-1} -\alpha_i(y_m -y_0)}{x_i -x_{i-1}- \alpha_i} \ne \frac{y_j -y_{j-1} -\alpha_j(y_m -y_0)}{x_j -x_{j-1}- \alpha_j}.
       \end{equation} 
Let $f$ be an affine FIF associated with the above data set and denote by $G_f$ its graph. Then, $\dim_H G_f = s$ where $s$ is the unique positive solution of 
\[
\sum\limits_{i=1}^{m} |\alpha_i| (x_i -x_{i-1})^{s-1}=1.
\]
\end{theorem}
 Note that Theorem \ref{Barany} implies that under the condition \eqref{condi} the box dimension of $G_f$ equals its Hausdorff dimension. 
       \begin{theorem}\label{mainthm4}
       Let $f \in \mathcal{C}(I)$ and $\beta \in (1,2)$. Then there exists a sequence of fractal functions converging uniformly to $f$ with their graphs having Hausdorff dimension $\beta.$
       \end{theorem}
       \begin{proof}
       We consider a sequence of data set $\triangle_n=\{(x_i,f(x_i)) \in I \times \mathbb{R}:i=0,1,2,\dots ,n\}$ with $0=x_0 < x_1< \dots
              < x_n=1$ and $x_i -x_{i-1}= \frac{1}{n}.$ Choose $\alpha_i = \alpha= \frac{1}{n^{2-\beta}}$ for every $i =1,2,\dots, n.$ Then we have $s=2 +\frac{\log(|\alpha|)}{\log n}=\beta.$ Moreover, $\sum\limits_{i=1}^{n}|\alpha_i|= n |\alpha|= n^{\beta -1} > 1.$ By Theorem \ref{Barany} it suffices to show that $f(x_i) -f(x_{i-1}) \ne f(x_j) -f(x_{j-1})$, for some $i \ne j$, in order to verify condition \eqref{condi}. For each $n \ge 2$, we define a data set $\tilde{\triangle}_n$ by 
              \begin{equation*}
              \tilde{\triangle}_n=
              \begin{cases}
               \triangle_n, ~~~ \text{if} ~ f(x_1) -f(x_0) \ne f(x_n) -f(x_{n-1}) \\
               \{(x_i,y_i): i= 0,1,2,\dots ,n\}, ~~\text{otherwise},
              \end{cases}
              \end{equation*}
              where $y_0= f(x_0)+\frac{1}{n}, ~ y_i=f(x_i)$ for $i =1,2,\dots ,n.$  Finally, we obtain a sequence $(g_n)$ of fractal interpolation functions generated by the data set $\tilde{\triangle}_n$ and the aforementioned scale vector $\alpha$ converging to $f$ and satisfying the desired condition.
       \end{proof}
%       \begin{remark}
%       Combining Theorems \ref{mainthm3}, \ref{mainthm4}, \ref{mainthm} and \ref{mainthm1}, we hope that our paper will find some applications in approximation theory.
%       \end{remark}
       \section{Restrictions and Extensions of Continuous Functions}
       In this section, we focus on some restrictions and extensions of continuous functions in regards to fractal dimensions. For this purpose, we need to state some known results.
      \begin{theorem}[\cite{Fal}, Theorem $4.10$]\label{falcomp}
      Let $A \subset \mathbb{R}^n$ be a Borel set such that $ 0 < \mathcal{H}^s(A) \le \infty.$ Then there exists a compact set $K \subset A$ such that $ 0< \mathcal{H}^s(K) <\infty.$
      \end{theorem}
      
In \cite{JP} the above result was also established for the packing dimension.
      \begin{theorem}\label{JPcomp}
            Let $A \subset \mathbb{R}^n$ be a Borel set such that $ 0 < \mathcal{P}^s(A) \le \infty.$ Then there exists a compact set $K \subset A$ such that $ 0< \mathcal{P}^s(K) <\infty.$
            \end{theorem}
      
      \begin{lemma}\label{compset}
      Let $A$ be a compact subset of $\mathbb{R}^n$, and $s \le \dim_H A$. Then there exists a compact set $K \subset A$ such that $ \dim_H K =s.$ The analogous result holds for the packing dimension.
      \end{lemma}
      \begin{proof}
      Suppose that $s< \dim_H A$. Using the definition of Hausdorff dimension, we have $\mathcal{H}^s(A)=\infty$. Theorem \ref{falcomp} produces a compact subset $K$ of $A$ satisfying $0 < \mathcal{H}^s (K) < \infty.$ Again using the definition of Hausdorff dimension, this implies that $\dim_H K =s.$ The case $s=\dim_H A$ is trivial. Thanks to Theorem \ref{JPcomp}, we have the same result for the packing dimension.
      \end{proof}
      \begin{theorem}
      Suppose $f\in \cC[0,1]$. Then, for each $0 \le \beta \le \dim_H G_f$, there exists a compact set $K \subset [0,1]$ such that $\dim G_f(K) =\beta$, where $G_f(K)=\{(x,f(x)): x \in K\} \subset \mathbb{R}^2.$ The same result holds for the packing dimension.
      \end{theorem}
      \begin{proof}
      For $\beta \le \dim G_f$. Using Lemma \ref{compset} we have a compact subset $K_1$ of $G_f$ such that $\dim K_1=\beta.$ We now show that there exists a compact set $K_2\subset [0,1]$ such that $G_f(K_2)=K_1.$ Define $K_2$ by $K_2:=\{x \in [0,1]: (x,f(x)) \in K_1\}$. If $(x_n)$ is a sequence in $K_2$ then $(x_n, f(x_n)) \in K_1 \subset G_f.$ By compactness of $K_1$ there exists a convergent subsequence of $\big((x_n, f(x_n)) \big)$. Denote this convergent subsequence again by $(x_n)$ and let $(x,f(x)) \in K_1$ be its limit. Hence, $(x_n)$ converges to $x$ and $x \in K_2$ completing the proof.  
      \end{proof}
      
      \begin{lemma}\label{newlem}
      For fixed $y_0,y_1 \in \mathbb{R}$ and $\beta \in [1,2]$, there exists $f \in \mathcal{C}(I)$ such that $f(0)=y_0,~f(1)=y_1$ and $\dim_H G_f=\beta.$
      \end{lemma}
      \begin{proof}
      In the light of Lemma \ref{densethm} we choose $h \in \mathcal{C}(I)$ with $h(0)=y_0$ and $\dim_H G_h=\beta.$ Define a Lipschitz mapping $g:[0,1] \to \mathbb{R}$ by $g(x)=(y_1 -h(1))x.$ Defining a map $f:[0,1]\to \mathbb{R}$ by $f=g+h$, Lemma \ref{lipdim} in turn yields the result.
      \end{proof}
        The next theorem is a modification of Proposition $2.3$ appeared in \cite{Liu}. For the convenience of the reader, we include the proof. 
      \begin{theorem}\label{extens}
       Let $X$ be a proper compact subset of $[0,1]$ and function $f : X \to \mathbb{R}$ be a continuous. Then for each $\max \{\dim_H G_f(X), 1\} \le \beta \le 2$, the function $f$ can be extended continuously to a continuous function $\tilde{f}:[0,1]\to \R$ such that 
\[
\dim_H G_{\tilde{f}}([0,1]) = \beta.
\]
The result also holds for packing dimension.
       \end{theorem}
       \begin{proof} 
       Let $X$ be a proper compact subset of $[0,1]$ and $f:X \to \mathbb{R}$ a continuous function. Now we consider the following possibilities:
       \begin{enumerate}
       	\item $0,1\in X.$
       	\item $0 \in X$ and $1 \notin X.$ 
       	\item $1 \in X$ and $0 \notin X.$
       	\item $0,1\notin X$
       \end{enumerate}
       We write $[0,1]\backslash X$ for each of the four cases above as follows:
       \begin{enumerate}
       \item $[0,1]\setminus X = \displaystyle\bigcup\limits_{i = 1}^{\infty}(a_i,b_i)$, with $a_i,b_i\in X$ for each $i \in \mathbb{N}.$
       \item $[0,1]\setminus X = \displaystyle\bigcup\limits_{i = 1}^{\infty}(a_i,b_i)\cup\{1\}$, with $a_i,b_i\in X$ for each $i \in \mathbb{N}.$
       \item $[0,1]\setminus X = \displaystyle\bigcup\limits_{i = 1}^{\infty}(a_i,b_i)\cup\{0\}$, with $a_i,b_i\in X$ for each $i \in \mathbb{N}.$
       \item $[0,1]\setminus X = \displaystyle\bigcup\limits_{i = 1}^{\infty}(a_i,b_i)\cup\{0,1\}$, with $a_i,b_i\in X$ for each $i \in \mathbb{N}.$
       \end{enumerate}
       By the finite stability of Hausdorff dimension (cf. \cite{Fal}), we claim that it is enough to deal with the first case. Applying Lemma \ref{newlem} for each intervals $[a_i,b_i]$, 
       we extend the function $f$ as follows:
       \begin{align*}
       \tilde{f}(x) := \begin{cases}
       f(x),&~~x\in X;\\ \\
       g_i(x),&~~x\in (a_i,b_i)~ \text{for some}~ i \in \mathbb{N},
       \end{cases}
       \end{align*}
       where $g_i(a_i)=f(a_i),g_i(b_i)=f(b_i)$ and $\dim_H G_{g_i}=\beta.$ 
       Clearly, $\tilde{f}$ is continuous on $[0,1].$ Using the countable stability of the Hausdorff dimension \eqref{stab}, it follows that 
       \begin{align*}
        \dim_H G_{\tilde{f}}([0,1]\backslash X) =\sup_{i \in \mathbb{N} } \{ \dim_H G_{\tilde{f}}((a_i,b_i)) \}
       =  \sup_{i \in \mathbb{N} } \{ \beta \} = \beta,
       \end{align*}
       and
       \begin{align*}
       \dim_H G_{\tilde{f}}([0,1])  &= \max\{ \dim_H G_{\tilde{f}}(X) ,\dim_H G_{\tilde{f}} ([0,1]\backslash X)  \}\\
       &= \max\{ \dim_H G_{f}(X) , \beta \}\\ & =\beta.
       \end{align*}
       Hence we obtain the result for the Hausdorff dimension. Since the packing dimension is also countably stable, the result for $\dim_P$ follows immediately.
       \end{proof}
   \section{Summary}   
   In this article we investigated a new notion of constrained approximation through fractal dimensions. Further, we constructed dimension preserving approximants to a prescribed function. In the last part of article, we introduced and investigated the restrictions to and extensions of continuous functions in terms of fractal dimensions.
 \subsection*{Acknowledgements}
   The first author expresses his gratitude to the University Grants
   Commission (UGC), India, for financial support, and to his supervisor Dr. P. Viswanathan for his support and encouragement.

\bibliographystyle{amsplain}

\end{document}